\documentclass[a4paper,12pt]{article}
\usepackage{amsmath,amssymb,amsthm,graphics,latexsym,amsfonts}
\usepackage{fancyhdr}
\usepackage{color}
\usepackage{graphics}
\usepackage{epstopdf}
\usepackage{amssymb}
\usepackage{cite}
\usepackage{hyperref}
\usepackage{tikz}
\usepackage{cleveref}

\hypersetup{
    colorlinks=true,
    linkcolor=cyan,
    filecolor=cyan,
    urlcolor=cyan,
    citecolor=cyan,
}

\title{\Large Isolation partitions in graphs}
	\author{ {Gang Zhang\,, Weiling Yang\footnotemark[1]\,, Xian'an Jin}\vspace{2mm}\\
	\small  School of Mathematical Sciences, Xiamen University,\\
	\small  Xiamen, Fujian 361005, P.R. China\\
}
\date{}

\newtheorem{theorem}{Theorem}[section]
\newtheorem{lemma}[theorem]{Lemma}
\newtheorem{corollary}[theorem]{Corollary}
\newtheorem{claim}{Claim}[section]

\newtheorem{conjecture}[theorem]{Conjecture}
\newtheorem{proposition}[theorem]{Proposition}
\newtheorem{problem}[theorem]{Problem}

\usepackage{indentfirst}

\newcommand{\vertex}{\node[vertex]}
\tikzstyle{vertex}=[circle, draw, inner sep=0pt, minimum size=6pt]

\begin{document}
	
\renewcommand{\thefootnote}{\fnsymbol{footnote}}
\footnotetext[1]{Corresponding author.\\
	\hangindent=1.8em E-mail addresses: gzh\_ang@163.com, ywlxmu@xmu.edu.cn, xajin@xmu.edu.cn.}
	
\maketitle

\small \noindent{\bfseries Abstract} Let $G$ be a graph and $k \geq 3$ an integer. A subset $D \subseteq V(G)$ is a $k$-clique (resp., cycle) isolating set of $G$ if $G-N[D]$ contains no $k$-clique (resp., cycle). In this paper, we prove that every connected graph with maximum degree at most $k$, except $k$-clique, can be partitioned into $k+1$ disjoint $k$-clique isolating sets, and that every connected claw-free subcubic graph, except 3-cycle, can be partitioned into four disjoint cycle isolating sets. As a consequence of the first result, every $k$-regular graph can be partitioned into $k+1$ disjoint $k$-clique isolating sets.


\section {\large Introduction}

All graphs considered in this paper are simple, finite and undirected. Let $G=(V,E)$ be a such graph. For a vertex $v \in V(G)$, we use $N_G(v)$ to denote the set of neighbors of $v$ in $G$, that is, $N_G(v)=\{u \in V(G): uv \in E(G)\}$. Let $N_G[v]=N_G(v) \cup \{v\}$. As usual, the sets $N_G(v)$ and $N_G[v]$ are called the {\it open} and {\it closed neighborhood of $v$} in $G$, respectively. For a subset $S \subseteq V(G)$, the {\it open} and {\it closed neighborhood of $S$} in $G$ are the sets $N_G(S):=\bigcup_{v \in S}N_G(v) \setminus S$ and $N_G[S]=N_G(S) \cup S$. Without causing ambiguity, we shall omit the subscript ``$G$'' from the notations above, for instance, $N(v)=N_G(v)$ simply. For a subset $S \subseteq V(G)$, let $G[S]$ and $G-S$ be the {\it subgraphs of $G$ induced by $S$} and {\it $V(G) \setminus S$}, respectively. The readers can be referred to \cite{Bondy2008} for more graph theory terminologies that are not explicitly described or defined here.

Let $P_n$, $C_n$, $K_n$ and $K_{1,n-1}$ denote the {\it $n$-path}, {\it $n$-cycle}, {\it $n$-clique} and {\it $n$-star}, respectively. A {\it tree} is a connected acyclic graph. A {\it claw-free} graph means it contains no induced $K_{1,3}$. For an integer $k \geq 0$, a graph $G$ is said to be {\it $k$-regular} if $d(v)=k$ for any vertex $v \in V(G)$. A {\it cubic} graph is a 3-regular graph. Let $\Delta(G)$ denote the {\it maximum degree} of a graph $G$. The graph $G$ is said to be {\it subcubic} if $\Delta(G) \leq 3$. For any integer $l \geq 1$, we denote $[l]:=\{1,2,\ldots,l\}$ simply.

\vspace{3mm}
The isolation of graphs is a natural generalization of domination, which was introduced by Caro and Hansberg \cite{Caro2017} in 2017.

\begin{itemize}
	\item A subset $D \subseteq V(G)$ is a {\it dominating set} of a graph $G$ if each vertex in $V(G) \setminus D$ is adjacent to at least one vertex in $D$, i.e., $N[D]=V(G)$ or $G-N[D] \cong (\emptyset,\emptyset)$. The {\it domination number} of $G$, denoted by $\gamma(G)$, is the minimum cardinality of a dominating set of $G$.
	
	\item A subset $D \subseteq V(G)$ is an {\it $\mathcal{H}$-isolating set} of a graph $G$ if $G-N[D]$ contains no member of $\mathcal{H}$ as a subgraph, where $\mathcal{H}$ is a family of connected graphs. The {\it $\mathcal{H}$-isolation number} of $G$, denoted by $\iota(G,\mathcal{H})$, is the minimum cardinality of an $\mathcal{H}$-isolating set of $G$. If $\mathcal{H}=\{H\}$, then we simply have the notation of {\it $H$-isolation} of $G$, and write $\iota(G,H)$ for short.
\end{itemize}

\noindent Note that $K_1$-isolation is the domination of a graph $G$, and that $\gamma(G)=\iota(G,K_1)$. A $K_2$-isolating set of a graph $G$ is simply called an {\it isolating set} of $G$, and we denote by $\iota(G)$ instead of $\iota(G,K_2)$ the {\it isolation number} of $G$. For any integer $k \geq 1$, $K_k$-isolation is called {\it $k$-clique isolation}, and we denote by $\iota(G,k)$ simply the {\it $k$-clique isolation number} of a graph $G$. Let $\mathcal{C}$ be the set of cycles of all lengths, i.e., $\mathcal{C}:=\{C_k: k \geq 3\}$. We call $\mathcal{C}$-isolation {\it cycle isolation}, and denote by $\iota_c(G)$ simply the {\it cycle isolation number} of a graph $G$.

\vspace{3mm}
In this paper, we study whether the vertices of a graph can be partitioned into a certain number of disjoint $\mathcal{H}$-isolating sets or not, specifically, $k$-clique and cycle isolating sets. Several known results are summarized as follows, and more research on isolation in graphs can be found in numerous literatures, e.g. \cite{Bartolo2024,Chen2023,Cui2024,Yan2022,Yin2024,Zhang2024}.

	\begin{theorem}(Ore \cite{Ore1962}).\label{th1.1}
	Every connected graph, except $K_1$, can be partitioned into two disjoint dominating sets.
    \end{theorem}

\begin{corollary}(Ore \cite{Ore1962}).
	If $G \ncong K_1$ is a connected graph of order $n$, then $$\gamma(G) \leq \frac{n}{2}.$$
\end{corollary}

	\begin{theorem}(Caro and Hansberg \cite{Caro2017}).\label{th1.3}
	If $G \notin \{K_2,C_5\}$ is a connected graph of order $n$, then $$\iota(G) \leq \frac{n}{3}.$$
    \end{theorem}

	\begin{theorem}(Borg, Fenech and Kaemawichanurat \cite{Borg&Fenech2020}).\label{th1.4}
	Let $k \geq 3$ be an integer. If $G \ncong K_k$ is a connected graph of order $n$, then $$\iota(G,k) \leq \frac{n}{k+1}.$$
\end{theorem}

	\begin{theorem}(Borg \cite{Borg2020}).\label{th1.5}
	If $G \ncong C_3$ is a connected graph of order $n$, then $$\iota_c(G) \leq \frac{n}{4}.$$
    \end{theorem}

	\begin{theorem}(Boyer and Goddard \cite{Boyer2024}).\label{th1.6}
	Every connected graph, except $K_2$ and $C_5$, can be partitioned into three disjoint isolating sets.
    \end{theorem}

Note that Theorem \ref{th1.3} can be regard as a corollary of Theorem \ref{th1.6}. So, we would like to propose the following two conjectures, and if they are true, then they will imply Theorems \ref{th1.4} and \ref{th1.5}, respectively.

	\begin{conjecture}\label{conj1.7}
	Let $k \geq 3$ be an integer. Every connected graph, except $K_k$, can be partitioned into $k+1$ disjoint $k$-clique isolating sets.
    \end{conjecture}

\begin{conjecture}\label{conj1.8}
	Every connected graph, except $C_3$, can be partitioned into four disjoint cycle isolating sets.
\end{conjecture}

As for Conjectures \ref{conj1.7} and \ref{conj1.8}, we shall prove the following results in this note.

\begin{theorem}\label{th1.9}
	Let $k \geq 3$ be an integer. Every connected graph $G$ with $\Delta(G) \leq k$, except $K_k$, can be partitioned into $k+1$ disjoint $k$-clique isolating sets.
\end{theorem}

\begin{corollary}
	Every $k$-regular graph can be partitioned into $k+1$ disjoint $k$-clique isolating sets.
\end{corollary}

\begin{theorem}\label{th1.11}
	Every connected claw-free subcubic graph, except $C_3$, can be partitioned into four disjoint cycle isolating sets.
\end{theorem}

\section {\large Clique isolation partitions}\label{sec2}

In this section, we prove Theorem \ref{th1.9}. We first define the following coloring.

	\begin{itemize}
	\item A {\it $(k+1)$-$k$-CI-coloring} of a graph $G$ is a mapping $c$ from $V(G)$ to the color set $[k+1]$ such that each color class is a $k$-clique isolating set of $G$.
	
	\item A graph $G$ is said to be {\it $(k+1)$-$k$-CI-colorable} if it admits a $(k+1)$-$k$-CI-coloring.
    \end{itemize}

To prove Theorem \ref{th1.9}, we only need to prove the following theorem.
\begin{theorem}\label{th2.1}
	Let $k \geq 3$ be an integer. Every connected graph $G$ with $\Delta(G) \leq k$, except $K_k$, is $(k+1)$-$k$-CI-colorable.
\end{theorem}

The following two lemmas are useful for the proof of Theorem \ref{th2.1}.

\begin{lemma}\label{lem2.2}
	Let $G$ be a connected graph and $S \subseteq V(G)$. If $G-S$ admits a $(k+1)$-$k$-CI-coloring $\left.c\right|_{G-S}$, and $G[S]$ admits a $(k+1)$-$k$-CI-coloring $\left.c\right|_{S}$ such that for each color class $D_S$ of $G[S]$, each component of $G[S]-N_{G[S]}[D_S]$ joins to $V(G)\setminus S$ by at most $k-2$ edges (especially, including the case of $G[S]-N_{G[S]}[D_S] \cong (\emptyset,\emptyset)$), then $c:=\left.c\right|_{S} \cup \left.c\right|_{G-S}$ is a $(k+1)$-$k$-CI-coloring of $G$, and thus, $G$ is $(k+1)$-$k$-CI-colorable.
\end{lemma}

\begin{proof}
	For each $i \in [k+1]$, let $D_S^{i}$ and $D^{i}_{G-S}$ be the $i$-th color classes of $G[S]$ and $G-S$, respectively. It suffices to prove that $D_S^{i} \cup D^{i}_{G-S}$ is a $k$-clique isolating set of $G$. Since each component of $G[S]-N_{G[S]}[D^{i}_S]$ joins to $V(G)\setminus S$ at most $k-2$ edges, we create no new $K_k$ between $G[S]-N_{G[S]}[D^{i}_S]$ and $(G-S)-N_{G-S}[D^{i}_{G-S}]$. So, $G-N[D_S^{i} \cup D^{i}_{G-S}]$ contains no $K_k$, and the result follows.
\end{proof}

	\begin{lemma}\label{lem2.3}
	Let $G_1,G_2,\ldots,G_s$ be the distinct components of a graph $G$. If for each $i \in [s]$, $G_i$ admits a $(k+1)$-$k$-CI-coloring $\left.c\right|_{G_i}$, then $c:=\cup_{i \in [s]}\left.c\right|_{G_i}$ is a $(k+1)$-$k$-CI-coloring of $G$. Conversely, if $G$ admits a $(k+1)$-$k$-CI-coloring $c$, then for each $i \in [s]$, $\left.c\right|_{G_i}$ is a $(k+1)$-$k$-CI-coloring of $G_i$.
    \end{lemma}

\begin{proof}
	For each $i \in [k+1]$, let $D_i$ be the $i$-th color class of $G$, and for each $j \in [s]$, let $D_i^j$ be the be the $i$-th color class of $G_j$. For the former, it suffices to prove that $D_i=\cup_{j \in [s]}D_i^j$ is a $k$-clique isolating set of $G$. For the latter, it suffices to prove that $D_i^j=D_i \cap V(G_j)$ is a $k$-clique isolating set of $G_j$. In fact, we can easily see $$G-N[D_i]=\bigcup_{j \in [s]}(G_j-N_{G_j}[D_i^j]).$$
	\noindent So, both arguments are obvious, and the result follows.
\end{proof}

	\begin{corollary}
	A graph $G$ is $(k+1)$-$k$-CI-colorable if and only if each component of $G$ is $(k+1)$-$k$-CI-colorable.
    \end{corollary}

\noindent\textbf{Proof of Theorem \ref{th2.1}.} Let $G \ncong K_k$ be a counterexample of minimum order $n$. Note that $k\geq 3$ and $\Delta(G) \leq k$. If $G$ contains no $K_k$, then, clearly, $G$ is $(k+1)$-$k$-CI-colorable, a contradiction. Assume that $G$ contains a $K_k$, and it is easy to verify that $n \geq k+1$. In what follows we present several claims describing some structural properties of $G$ which culminate in the implication of its non-existence.

\vspace{3mm}
A {\it separating} $K_k$ of $G$ is a $k$-clique $H$ of $G$ such that $G-V(H)$ is disconnected.

\begin{claim}\label{claim2.1}
	Each $K_k$ in $G$ is separating.
\end{claim}

\begin{proof}
	Suppose to the contrary that there exists a non-separating $K_k$ in $G$, denoted by $H$, that is, $G-V(H)$ is connected. Since $n \geq k+1$ and $G$ is connected, $G-V(H) \ncong (\emptyset,\emptyset)$. Let $G':=G-V(H)$, and $uv \in E(G)$ for some $u \in V(H)$ and $v \in V(G')$. If $G' \cong K_k$, then we define a coloring $c$ of $G$ as follows: $c(u)=k+1$; $c(v)=k$; use $\{1,2,\ldots,k-1\}$ to color $V(H) \setminus \{u\}$ such that each vertex has a different color; and use $\{1,2,\ldots,k-2,k+1\}$ to color $V(G')\setminus \{v\}$ such that each vertex has a different color. It is easy to check that for each $i \in [k+1]$, the color class $D_{i}$ of $G$ is a $k$-clique isolating set of $G$. Therefore, $c$ is a $(k+1)$-$k$-CI-coloring of $G$, and $G$ is $(k+1)$-$k$-CI-colorable.
	
	\begin{figure}[h!]
		\begin{center}
			\begin{tikzpicture}[scale=.48]
				\tikzstyle{vertex}=[circle, draw, inner sep=0pt, minimum size=6pt]
				\tikzset{vertexStyle/.append style={rectangle}}
				
				\vertex (1) at (-9,0) [scale=0.35,fill=black] {};

				\vertex (2) at (-12,3) [scale=0.35,fill=black] {};			
				\node ($u$) at (-12.5,3) {$u$};
				
				\vertex (3) at (-12,-3) [scale=0.35,fill=black] {};
				
				\node ($H$) at (-10.7,0) {$H$};
				
				\draw (0,0) circle (3cm);
				\node ($G'$) at (0,0) {$G'$};
				
				\coordinate (4) at (225:3cm);
				\fill (4) circle (3pt);
				
				\coordinate (6) at (135:3cm);
				\fill (6) circle (3pt);
				\node ($v$) at (135:2.3cm) {$v$};
				
				\vertex (7) at (-5,-2.13) [scale=0.35,fill=black] {};			
				\node ($v'$) at (-5,-2.63) {$v'$};
				
				\node ($H'$) at (-8.5,-1.5) {$H'$};

				\path
				(1) edge (2)
				(1) edge (3)
				(2) edge (3)
				(2) edge (6)
				(4) edge (7)
				;
				
				\path
				(1) edge (7) [dashed]
				(3) edge (7)
				;

			\end{tikzpicture}
		\end{center}
		\par {\footnotesize \centerline{{\bf Fig. 1.} ~The unique $k$-clique $H'$ in $G-N[D_{k+1}]$ where $k=3$.\hypertarget{Fig1}}}
	\end{figure}
	
	If $G' \ncong K_k$, then by the minimality of $G$, $G'$ admits a $(k+1)$-$k$-CI-coloring $\left.c\right|_{G'}$. W.l.o.g., we may assume that $\left.c\right|_{G'}(v)=k+1$. Now we use $\{1,2,\ldots,k\}$ to color $V(H)$ such that each vertex has different color, and denote $\left.c\right|_{H}$ and $c:=\left.c\right|_{G'} \cup \left.c\right|_{H}$ by the present colorings of $H$ and $G$. Clearly, for each $i \in [k]$, the color class $D_i$ of $G$ is a $k$-clique isolating set of $G$, and the color class $D_{k+1}$ is not a $k$-clique isolating set of $G$ (otherwise, $c$ is a $(k+1)$-$k$-CI-coloring of $G$) if and only if $G-N[D_{k+1}]$ contains a $k$-clique. Since $\Delta(G) \leq k$, the unique $k$-clique of $G-N[D_{k+1}]$ only can be formed by $(V(H) \setminus \{u\}) \cup \{v'\}$ (see Fig. \hyperlink{Fig1}{1}), where $v'$ is a vertex of $G'$ distinct from $v$. Note that $c(v') \neq k+1$ (otherwise, $D_{k+1}$ is a $k$-clique isolating set of $G$). W.l.o.g., we may assume $c(v')=k$. Then we recolor $V(H)$ by $\{1,2,\ldots,k-1,k+1\}$ such that each vertex has different color, and denote $\left.c'\right|_{H}$ and $c':=\left.c\right|_{G'} \cup \left.c'\right|_{H}$ by the present colorings of $H$ and $G$. One can verify that $c'$ is a $(k+1)$-$k$-CI-coloring of $G$ in this situation. Therefore, $G$ is always $(k+1)$-$k$-CI-colorable.
\end{proof}

Let $K_{k}^+$ be the graph obtained from $K_k$ by adding a pendant edge to it.

\begin{claim}\label{claim2.2}
	Each $K_k^+$ in $G$ is induced.
\end{claim}

\begin{proof}
	Suppose to the contrary that there exists a non-induced $K_k^+$ in $G$, denoted by $H$. There exists an edge $uv \in E(H^+)\setminus E(H)$ where $d_{H}(u)=1$, $v \in V(H) \setminus \{u\}$ and $H^+:=G[V(H)]$. Let $G':=G-V(H^+)=G-V(H)$. If $G' \cong (\emptyset,\emptyset)$, then $G$ is $(k+1)$-$k$-CI-colorable obviously. Assume that $G' \ncong (\emptyset,\emptyset)$. By Claim \ref{claim2.1}, we know that $G'$ contains no $K_k$-component. By Lemma \ref{lem2.3} and the minimality of $G$, $G'$ admits a $(k+1)$-$k$-CI-coloring $\left.c\right|_{G'}$. Now we use $\{1,2,\ldots,k+1\}$ to color $V(H^+)$ such that each vertex has a different color, and denote $\left.c\right|_{H^+}$ and $c:=\left.c\right|_{G'} \cup \left.c\right|_{H^+}$ by the present colorings of $H^+$ and $G$. Since $\Delta(G) \leq k$ and $d_H(u) \geq 2$, $u$ joins to $V(G')$ by at most $k-2$ edges. By Lemma \ref{lem2.2} (regrading $V(H^+)$ as $S$), $c$ is a $(k+1)$-$k$-CI-coloring of $G$, and thus, $G$ is $(k+1)$-$k$-CI-colorable.
\end{proof}

A {\it double $k$-clique} $DK_k$ is the connected graph obtained from two vertex-disjoint $K_k$ by adding an edge between them.

	\begin{figure}[h!]
	\begin{center}
		\begin{tikzpicture}[scale=.42]
			\tikzstyle{vertex}=[circle, draw, inner sep=0pt, minimum size=6pt]
			\tikzset{vertexStyle/.append style={rectangle}}
			
			\vertex (1) at (-2,7) [scale=0.35,fill=black] {};

			\vertex (2) at (-5,4) [scale=0.35,fill=black] {};
			
			\vertex (3) at (-5,10) [scale=0.35,fill=black] {};
			
			\vertex (4) at (2,7) [scale=0.35,fill=black] {};
			
			\vertex (5) at (5,4) [scale=0.35,fill=black] {};
			
			\vertex (6) at (5,10) [scale=0.35,fill=black] {};

			\node ($H$) at (0,8.5) {$H$};
			
			\node ($K_3^+$) at (7,2) {$K_3^+$};
			

			\draw (0,0) circle (3cm);
			\node ($G'$) at (0,0) {$G'$};
			
			\coordinate (7) at (180:3cm);
			\fill (7) circle (3pt);

			\coordinate (8) at (0:3cm);
			\fill (8) circle (3pt);
			
			\vertex (9) at (-5,0) [scale=0.35,fill=black] {};
			
			\vertex (10) at (5,0) [scale=0.35,fill=black] {};

			\path
			(1) edge (2)
			(1) edge (3)
			(2) edge (3)
			(1) edge (4)
			(4) edge (5)
			(4) edge (6)
			(5) edge (6)
			(8) edge (10)
			;
			
			\draw (-5,10)..controls (-7,7) and (-7,3) ..(9) [dashed];
			
			\draw (5,10)..controls (7,7) and (7,3) ..(10) [dashed];
			
			\path
			(2) edge (9) [dashed]
			(5) edge (10)
			(9) edge (7)
			;

		\end{tikzpicture}
	\end{center}
	\par {\footnotesize \centerline{{\bf Fig. 2.} ~There exists a non-induced $K_k^+$ in $G$ where $k=3$.\hypertarget{Fig2}}}
\end{figure}

\begin{claim}\label{claim2.3}
	$G$ contains no $DK_k$.
\end{claim}

\begin{proof}
	Suppose to the contrary that there exists a $DK_k$ in $G$, denoted by $H$. Let $G':=G-V(H)$. We straightforward color $V(H)$ as in the proof (for the situation of $G' \cong K_k$) of Claim \ref{claim2.1}, and denote $\left.c\right|_{H}$ by present coloring of $H$. If $G' \cong (\emptyset,\emptyset)$, then by Claim \ref{claim2.1}, we know that $G$ is $(k+1)$-$k$-CI-colorable. If $G' \ncong (\emptyset,\emptyset)$, then by Claim \ref{claim2.1}, Lemma \ref{lem2.3} and the minimality of $G$, $G'$ contains no $K_k$-component, and $G'$ admits a $(k+1)$-$k$-CI-coloring $\left.c\right|_{G'}$. Let $c:=\left.c\right|_{G'} \cup \left.c\right|_{H}$. Since $k \geq 3$, $\Delta(G) \leq k$ and by Claim \ref{claim2.2}, we create no new $K_k$ between $H-N_{H}[D_H^i]$ and $G'-N_{G'}[D_{G'}^i]$, where $i \in [k+1]$, and $D_H^i$ and $D_{G'}^i$ are the $i$-th color classes of $H$ and $G'$; in this situation, the worst is that for some $j \in \{k,k+1\}$, there exists a non-induced $K_k^+$ in $G$ formed by one $K_k$ of the double $k$-clique $H$ and a vertex of $G'-N_{G'}[D_{G'}^j]$ (see Fig. \hyperlink{Fig2}{2}), which contradicts Claim \ref{claim2.2}. Thus, $c$ is a $(k+1)$-$k$-CI-coloring of $G$, and $G$ is $(k+1)$-$k$-CI-colorable.
\end{proof}

\begin{claim}\label{claim2.4}
	$G$ contains no $K_k$.
\end{claim}

\begin{proof}
	Suppose to the contrary that there exists a $K_k$ in $G$, denoted by $H$. Since $n \geq k+1$, there exists a vertex $u \in V(G) \setminus V(H)$ such that $uv \in E(G)$ for some $v \in V(H)$. The set $V(H)\cup \{u\}$ forms a $K_k^+$ in $G$, denoted by $H^+$. By Claim \ref{claim2.2}, $uv' \notin E(G)$ for each $v' \in V(H)$. Let $G':=G-V(H^+)$. By Claim \ref{claim2.1}, Lemma \ref{lem2.3} and the minimality of $G$, $G'$ contains no $K_k$-component, and $G'$ admits a $(k+1)$-$k$-CI-coloring $\left.c\right|_{G'}$. Now we use $\{1,2,\ldots,k+1\}$ to color $V(H^+)$ such that each vertex has a different color, and denote $\left.c\right|_{H^+}$ and $c:=\left.c\right|_{G'} \cup \left.c\right|_{H^+}$ by the present colorings of $H^+$ and $G$. W.l.o.g., assume that $c(u)=k+1$ and $c(v)=k$. It is clear that the color class $D_k$ is a $k$-clique isolating set of $G$. By the same reason noted in the proof of Claim \ref{claim2.3}, we know that $D_{k+1}$ is a $k$-clique isolating set of $G$. For each $i \in [k]$, the color class $D_i$ of $G$ is not a $k$-clique isolating set of $G$ if and only if $G-N[D_i]$ contains a $k$-clique $H'$ (only can be formed by $u$ and $k-1$ vertices of $G'-(N[D_i] \cap V(G'))$. However, there is a $DK_k$ in $G$ formed by $V(H)$ and $V(H')$, a contradiction to Claim \ref{claim2.3}. Hence, the color class $D_i$ is also a $k$-clique isolating set of $G$, and $c$ is $(k+1)$-$k$-CI-coloring of $G$ and $G$ is $(k+1)$-$k$-CI-colorable.
\end{proof}

	By Claim \ref{claim2.4}, $G$ contains no a $K_k$, contradicting the assumption in the beginning of the proof that $G$ must contain a $K_k$. Therefore, the minimal counterexample $G$ does not exist. This completes the proof of Theorem \ref{th2.1}. \qed

\section {\large Cycle isolation partitions}

In this section, we prove Theorem \ref{th1.11}. Define a coloring as follows.

\begin{itemize}
	\item A {\it 4-CI-coloring} of a graph $G$ is a mapping $c$ from $V(G)$ to the color set $[4]$ such that each color class is a cycle isolating set of $G$.
	
	\item A graph $G$ is said to be {\it 4-CI-colorable} if it admits a 4-CI-coloring.
\end{itemize}

We only need to prove the following theorem.

\begin{theorem}\label{th3.1}
	Every connected claw-free subcubic graph, except $C_3$, is 4-CI-colorable.
\end{theorem}

\begin{lemma}\label{lem3.2}
	Let $G$ be a connected graph and $S \subseteq V(G)$. If $G-S$ admits a 4-CI-coloring $\left.c\right|_{G-S}$, and $G[S]$ admits a 4-CI-coloring $\left.c\right|_{S}$ such that for each color class $D_S$ of $G[S]$, each component of $G[S]-N_{G[S]}[D_S]$ joins to $V(G)\setminus S$ by at most one edge, then $c:=\left.c\right|_{S} \cup \left.c\right|_{G-S}$ is a 4-CI-coloring of $G$, and thus, $G$ is 4-CI-colorable.
\end{lemma}

\begin{proof}
	The proof is similar to the proof of Lemma \ref{lem2.2}. For each $i \in [4]$, let $D_S^i$ and $D_{G-S}^i$ be defined as in Lemma \ref{lem2.2}. Since each component of $G[S]-N_{G[S]}[D^{i}_S]$ joins to $V(G)\setminus S$ at most one edge, we create no new cycle through the edges between $G[S]-N_{G[S]}[D^{i}_S]$ and $(G-S)-N_{G-S}[D^{i}_{G-S}]$. Thus, $G-N[D_S^{i} \cup D^{i}_{G-S}]$ contains no cycle and $D_S^{i} \cup D^{i}_{G-S}$ is a cycle isolating set of $G$. The result follows.
\end{proof}

\begin{lemma}\label{lem3.3}
	Let $G_1,G_2,\ldots,G_s$ be the distinct components of a graph $G$. If for each $i \in [s]$, $G_i$ admits a 4-CI-coloring $\left.c\right|_{G_i}$, then $c:=\cup_{i \in [s]}\left.c\right|_{G_i}$ is a 4-CI-coloring of $G$. Conversely, if $G$ admits a 4-CI-coloring $c$, then for each $i \in [s]$, $\left.c\right|_{G_i}$ is a 4-CI-coloring of $G_i$.
\end{lemma}

\begin{proof}
	Obvious.
\end{proof}

\begin{lemma}\label{lem3.4}
	If $G$ is a tree or cycle, then Theorem \ref{th3.1} is true.
\end{lemma}

\begin{proof}
	If $G$ is a tree, then $G$ contains no cycles, and clearly, $G$ is 4-CI-colorable. If $G$ is a cycle, then $G \cong C_n$. Note that $n \neq 3$. Then $n \geq 4$. Let $v_1,v_2,v_3,v_4 \in V(G)$. Give a coloring $c$ of $G$ such that $c(v_i)=i$ for each $i \in [4]$, and for other vertices of $G$, we color them arbitrarily. It is easy to verify that each color class of $G$ is a cycle isolating set of $G$, and thus, $G$ is 4-CI-colorable. The result follows.
\end{proof}

\noindent\textbf{Proof of Theorem \ref{th3.1}.} Let $G \ncong C_3$ be a counterexample of minimum order $n$. Note that $\Delta(G) \leq 3$ and $G$ is claw-free. If $\Delta(G) \leq 2$, then $G \in \{P_n,C_n\}$. By Lemma \ref{lem3.4}, then $G$ is 4-CI-colorable. Assume that $\Delta(G)=3$. Then $n \geq 4$ and $G$ must contain a $C_3$. If $G \cong K_4$, then we use [4] to color $G$ such that each vertex of $G$ has a different color. It is easy to verify that $G$ is 4-CI-colorable. Assume that $G \ncong K_4$. The following several claims imply the non-existence of $G$.

\begin{claim}\label{claim3.1}
	$G$ contains no cycle of length a multiple of 4.
\end{claim}

\begin{proof}
	Suppose to the contrary that there exists a cycle $C$ of length a multiple of 4 in $G$. Clearly, $G-V(C)$ may be disconnected and contain some $C_3$-components. As shown in Fig. \hyperlink{Fig3}{3}, let $s$ be the number of these $C_3$-components, and $\{v_1^i,v_2^i,v_3^i\}$ be the vertex set of the $i$-th $C_3$-component where $i \in [s]$. W.l.o.g., we may assume that $uv_1^i \in E(G)$ for some $u \in V(C)$. Let $S:=V(C) \cup \bigcup_{i \in [s]}\{v_1^i,v_2^i,v_3^i\}$ and $G':=G-S$. By Lemma \ref{lem3.3} and the minimality of $G$, $G'$ admits a 4-CI-coloring $\left.c\right|_{G'}$. Define the coloring $\left.c\right|_{S}$ of $S$ as follows: we color $C$ by repeating 1,2,3,4; if $\left.c\right|_{S}(u)=j$ with $j \in [4]$, then we color $\{v_1^i,v_2^i,v_3^i\}$ by $[4]\setminus \{j\}$ such that each vertex has a different color (if $\{v_1^i,v_2^i,v_3^i\}$ has more than one neighbor in $V(C)$, we only choose one of them as the vertex $u$). Denote $c:=\left.c\right|_{G'} \cup \left.c\right|_{S}$ by the present coloring of $G$. It is easy to verify that for each $i \in [4]$, the color class $D_S^i$ of $G[S]$ is a cycle isolating set of $G[S]$. Since $\Delta(G) \leq 3$ and by Lemma \ref{lem3.2}, $c$ is a 4-CI-coloring of $G$, and thus, $G$ is 4-CI-colorable. The result follows.
\end{proof}

	\begin{figure}[h!]
	\begin{center}
		\begin{tikzpicture}[scale=.48]
			\tikzstyle{vertex}=[circle, draw, inner sep=0pt, minimum size=6pt]
			\tikzset{vertexStyle/.append style={rectangle}}
			
			\vertex (1) at (-14.5,0) [scale=0.35,fill=black] {};
			\node ($v_1^1$) at (-14.5,-0.7) {$v_1^1$};
			
			\vertex (2) at (-17.5,2) [scale=0.35,fill=black] {};
			\node ($v_3^1$) at (-17.5,2.7) {$v_3^1$};

			\vertex (3) at (-17.5,-2) [scale=0.35,fill=black] {};
			\node ($v_2^1$) at (-17.5,-2.7) {$v_2^1$};
			
			\node ($C$) at (-10.5,2) {$C$};
			
			\vertex (7) at (-5.5,0) [scale=0.35,fill=black] {};
			\vertex (8) at (-11.5,0) [scale=0.35,fill=black] {};
			\node ($u$) at (-11.5,-0.7) {$u$};
			
			\vertex (9) at (-8.5,3) [scale=0.35,fill=black] {};
			\vertex (10) at (-8.5,-3) [scale=0.35,fill=black] {};
			
			\vertex (4) at (-2.5,0) [scale=0.35,fill=black] {};
			
			\draw (0,0) circle (2.5cm);
			\node ($G'$) at (0,0) {$G'$};
			
			\path
			(2) edge (3)
			(1) edge (2)
			(1) edge (3)
			(1) edge (8)
			(7) edge (9)
			(7) edge (10)
			(8) edge (9)
			(8) edge (10)
			(9) edge (10)
			(4) edge (7)
			;

		\end{tikzpicture}
	\end{center}
	\par {\footnotesize \centerline{{\bf Fig. 3.} ~The case that $G$ contains a $C_4$ where $s=1$.\hypertarget{Fig3}}}
\end{figure}

\begin{claim}\label{claim3.2}
	$G$ contains no induced cycle of length $4t+3$ for any integer $t \geq 1$.
\end{claim}

\begin{proof}
	Suppose to the contrary that there exists an induced cycle $C$ of length $4t+3$ for some integer $t \geq 1$ in $G$. By Lemma \ref{lem3.4}, $G \neq C$. Hence, there exists a vertex $v$ in $V(G)\setminus V(C)$ such that $uv \in E(G)$ for some $u \in V(C)$. Let $N_C(u)=\{u_1,u_2\}$. Since $C$ is induced and $G$ is claw-free, $u_iv \in E(G)$ for some $i \in [2]$. Then we get a cycle of length $4t+4$ formed by $V(C) \cup \{v\}$, a contradiction to Claim \ref{claim3.1}.
\end{proof}

\begin{claim}\label{claim3.3}
	$G$ contains no induced cycle $C$ of length at least 4, where $C$ joins to $G-V(C)$ by at most two edges.
\end{claim}

\begin{proof}
	Suppose to the contrary that there exists an induced cycle $C$ such that $C$ joins to $G-V(C)$ by at most two edges. By Lemma \ref{lem3.4}, $G \neq C$. By Claims \ref{claim3.1} and \ref{claim3.2}, $|V(C)|=4t+i$ for some  $i \in [2]$ and integer $t \geq 1$. Let $G':=G-V(C)$. Since $C$ is induced and $G$ is claw-free, $C$ joins to $G'$ by exactly two edges $\{u_1v,u_2v\}$ in $G$, where $u_1,u_2$ are two consecutive vertices of $C$ and $v \in V(G')$ with $N(V(C))=\{v\}$. Since $\Delta(G) \leq 3$, $G' \ncong C_3$. Clearly, $G'$ is connected. By the minimality of $G$, $G'$ admits a 4-CI-coloring $\left.c\right|_{G'}$. We define a coloring $\left.c\right|_{C}$ of $C$ as follows: if $i=1$, then we use repeating 1,2,3,4 followed by an extra 3 (starting with $u_1$ and ending with $u_2$) to color $C$; if $i=2$, then we use repeating 1,2,3,4 followed by an extra 3,4 (starting with $u_1$ and ending with $u_2$) to color $C$. By Lemma \ref{lem3.2}, $c:=\left.c\right|_{C} \cup \left.c\right|_{G'}$ is a 4-CI-coloring of $G$, and $G$ is 4-CI-colorable.
\end{proof}

\begin{claim}\label{claim3.4}
	$G$ contains no induced cycle of length $4t+2$ for any integer $t \geq 1$.
\end{claim}

\begin{proof}
	Suppose to the contrary that there exists an induced cycle $C$ of length $4t+2$ for some integer $t \geq 1$ in $G$. Let $G':=G-V(C)$. By Claim \ref{claim3.3}, $C$ joins to $G'$ by at least three edges. Since $\Delta(G) \leq 3$, $G$ is claw-free and $C$ is induced, either there exists a vertex $v \in V(G')$ such that $vu_1,vu_2,vu_3 \in E(G)$, where $u_1,u_2,u_3$ are three consecutive vertices of $C$, or there exist at least two vertices $v_1,v_2 \in V(G)$ such that $u_1^1v_1,u_2^1v_1,u_1^2v_2,u_2^2v_2 \in E(G)$, where for each $i \in [2]$, $u_1^i,u_2^i$ are two consecutive vertices of $C$, as shown in Fig \hyperlink{Fig4}{4}. The first situation contradicts Claim \ref{claim3.1} since $\{v,u_1,u_2,u_3\}$ forms a cycle of length 4. The second situation also contradicts Claim \ref{claim3.1} since $V(C) \cup \{v_1,v_2\}$ forms a cycle of length $4t+4$.
\end{proof}

	\begin{figure}[h!]
	\begin{center}
		\begin{tikzpicture}[scale=.48]
			\tikzstyle{vertex}=[circle, draw, inner sep=0pt, minimum size=6pt]
			\tikzset{vertexStyle/.append style={rectangle}}
			
			\vertex (1) at (0,0) [scale=0.35,fill=black] {};
			
			\vertex (2) at (0,7) [scale=0.35,fill=black] {};
			\node ($u_1$) at (0.1,6.3) {$u_1$};
			
			\vertex (3) at (2,2) [scale=0.35,fill=black] {};
			\node ($u_2$) at (1.5,4.7) {$u_2$};		
			
			\vertex (4) at (2,5) [scale=0.35,fill=black] {};
			\node ($u_3$) at (1.5,2.3) {$u_3$};
			
			\vertex (5) at (-2,2) [scale=0.35,fill=black] {};	

            \vertex (6) at (-2,5) [scale=0.35,fill=black] {};
            
            \vertex (13) at (3.5,6) [scale=0.35,fill=black] {};
            \node ($v$) at (3.9,6) {$v$};
            
            \vertex (7) at (10,0) [scale=0.35,fill=black] {};
            \node ($u_2^2$) at (10,1) {$u_2^2$};
            
            \vertex (8) at (10,7) [scale=0.35,fill=black] {};
            \node ($u_1^1$) at (10.1,6) {$u_1^1$};
            
            \vertex (9) at (12,2) [scale=0.35,fill=black] {};
            \node ($u_2^1$) at (11.4,2.3) {$u_2^1$};
            
            \vertex (10) at (12,5) [scale=0.35,fill=black] {};
            \node ($u_1^2$) at (11.4,4.7) {$u_1^2$};
            
            \vertex (11) at (8,2) [scale=0.35,fill=black] {};		
            
            \vertex (12) at (8,5) [scale=0.35,fill=black] {};
            
            \vertex (14) at (12,0) [scale=0.35,fill=black] {};
            \node ($v_2$) at (12.6,0.2) {$v_2$};
            
            \vertex (15) at (12,7) [scale=0.35,fill=black] {};
            \node ($v_1$) at (12.6,6.8) {$v_1$};
			
			\node ($C$) at (0,3.5) {$C$};
			
			\node ($C$) at (10,3.5) {$C$};
			
			\path
			(1) edge (3)
			(1) edge (5)
			(3) edge (4)
			(6) edge (5)
			(2) edge (6)
			(2) edge (4)
			(13) edge (4)
			(13) edge (3)
			(13) edge (2)
			
			(7) edge (9)
			(7) edge (11)
			(9) edge (10)
			(12) edge (11)
			(8) edge (12)
			(8) edge (10)
			(14) edge (7)
			(14) edge (9)
			(15) edge (8)
			(15) edge (10)
			;
			
			\path
			(1) edge (0,-1.2) [dashed]
			(5) edge (-3.2,2)
			(6) edge (-3.2,5)
			(15) edge (13,8)
			(14) edge (13,-1)
			(11) edge (6.8,2)
			(12) edge (6.8,5)
			;
			
		\end{tikzpicture}
	\end{center}
	\par {\footnotesize \centerline{{\bf Fig. 4.} ~The possible cases that $G$ contains an induced $C_6$.\hypertarget{Fig4}}}
\end{figure}

\begin{claim}\label{claim3.5}
	$G$ contains no induced cycle of length $4t+1$ for any integer $t \geq 2$.
\end{claim}

\begin{proof}
	Suppose to the contrary that there exists an induced cycle $C$ of length $4t+1$ for some integer $t \geq 2$ in $G$. Let $G':=G-V(C)$. As noted in the proof of Claim \ref{claim3.4}, there exists no vertex $v$ in $G'$ such that $v$ joins to three consecutive vertices of $C$ (by Claim \ref{claim3.1}). Moreover, there exist at most two vertices $v_1,v_2$ in $G'$ such that for each $i \in [2]$, $v_i$ joins to two consecutive vertices $u_1^i,u_2^i$ of $C$ (also by Claim \ref{claim3.1}). By Claim \ref{claim3.3}, $C$ joins to $G'$ by at least three edges. Since $\Delta(G) \leq 3$, $G$ is claw-free and $C$ is induced, we know that $C$ joins to $G'$ by exactly four edges, and $N(V(C))=\{v_1,v_2\}$ where $v_1,v_2$ are defined as above. It is clear that $G'$ may be disconnected but contains no $C_3$-components. By Lemma \ref{lem3.3} and the minimality of $G$, $G'$ admits a 4-CI-coloring $\left.c\right|_{G'}$. We define a coloring $\left.c\right|_{C}$ of $C$ as follows: if $u_1^1,u_2^1,u_1^2,u_2^2$ are four consecutive vertices (clockwise) of $C$, then we first use repeating 1,2,3,4 followed by an extra 3 (starting with $u_2^1$ and ending with $u_1^1$) to color $C$, and then exchange the colors of $u_2^2$ and $u_3$, where $u_3$ is the next vertex of $u_2^2$ on $C$; otherwise, we straightforward use repeating 1,2,3,4 followed by an extra 3 (starting with $u_2^1$ and ending with $u_1^1$) to color $C$. For both situations, by Lemma \ref{lem3.2}, $c:=\left.c\right|_{C} \cup \left.c\right|_{G'}$ is a 4-CI-coloring of $G$, and $G$ is 4-CI-colorable.
\end{proof}

	\begin{figure}[h!]
	\begin{center}
		\begin{tikzpicture}[scale=.48]
			\tikzstyle{vertex}=[circle, draw, inner sep=0pt, minimum size=6pt]
			\tikzset{vertexStyle/.append style={rectangle}}
			
			\vertex (1) at (0,0) [scale=0.35,fill=black] {};
			\node ($v_2$) at (0,0.7) {$v_2$};
			\node ($u$) at (0.1,-0.7) {$(u)$};
			
			\vertex (2) at (0,7) [scale=0.35,fill=black] {};
			\node ($u_1^1$) at (0,6.1) {$u_1^1$};
			
			\vertex (3) at (2,2) [scale=0.35,fill=black] {};
			\node ($u_1^2$) at (1.5,2.55) {$u_1^2$};
			
			\vertex (4) at (2,5) [scale=0.35,fill=black] {};
			\node ($u_2^1$) at (1.5,4.6) {$u_2^1$};
			
			\vertex (5) at (-2,2) [scale=0.35,fill=black] {};
			\node ($u_2^2$) at (-1.5,2.55) {$u_2^2$};
			
			\vertex (6) at (-2,5) [scale=0.35,fill=black] {};
			\node ($u_3$) at (-1.5,4.6) {$u_3$};
			
			\vertex (7) at (2,7) [scale=0.35,fill=black] {};
			\node ($v_1$) at (2.6,6.7) {$v_1$};
			
			\node ($C$) at (0,4) {$C$};
			
			\node ($C'$) at (-1.5,6.5) {$C'$};
			
			\vertex (8) at (-6,5) [scale=0.35,fill=black] {};
			\vertex (9) at (-10,5) [scale=0.35,fill=black] {};
			\vertex (10) at (-8,2) [scale=0.35,fill=black] {};
			
			\node ($H$) at (-8,3.7) {$H$};
			
			\draw (8,3.5) circle (2cm);
			\node ($G''$) at (8,3.7) {$G''$};
			
			\vertex (11) at (8,5.5) [scale=0.35,fill=black] {};
			
			\path
			(1) edge (3)
			(1) edge (5)
			(3) edge (4)
			(6) edge (5)
			(2) edge (6)
			(2) edge (4)
			(3) edge (5)
			(2) edge (7)
			(4) edge (7)
			(8) edge (9)
			(9) edge (10)
			(8) edge (10)
			;
			
			\draw (2,7)..controls (4,8) and (6,8) ..(8,5.5);
			
			\draw (0,0)..controls (-3,-1) and (-6,-1) ..(10);
			
		\end{tikzpicture}
	\end{center}
	\par {\footnotesize \centerline{{\bf Fig. 5.} ~The case that $G$ contains an induced $C_5$.\hypertarget{Fig5}}}
\end{figure}

\begin{claim}\label{claim3.6}
	$G$ contains no induced cycle of length $5$.
\end{claim}

\begin{proof}
	Suppose to the contrary that there exists an induced cycle $C$ of length $5$ in $G$. Let $G':=G-V(C)$. As noted in the proof of Claim \ref{claim3.5}, there exist exactly two vertices $v_1,v_2$ in $G'$ such that for each $i \in [2]$, $v_i$ joins to two consecutive vertices $u_1^i,u_2^i$ of $C$, and $C$ joins to $G'$ by exactly four edges $u_1^1v_1,u_2^1v_1,u_1^2v_2,u_2^2v_2$ in $G$. Since $|V(C)|=5$, $u_1^1,u_2^1,u_1^2,u_2^2$ are four consecutive vertices (clockwise) of $C$. Let $u_3$ be the fifth vertex of $C$. As in Fig. \hyperlink{Fig5}{5}, $v_1u_2^1u_1^2v_2u_2^2u_3u_1^1v_1$ is a cycle $C'$ of length 7 in $G$, and $G-V(C')$ contains at most two components that may be isomorphic to $C_3$. Let $S$ be the union of $V(C')$ and the vertices of these possible $C_3$-components of $G-V(C')$, and let $G'':=G-S$. By Lemma \ref{lem3.3} and the minimality of $G$, $G''$ admits a 4-CI-coloring $\left.c\right|_{G''}$. Define a coloring $\left.c\right|_{S}$ of $S$ as follows: we use 1,2,3,4,1,2,4 to color $C'$ (starting with $u_1^1$ and ending with $u_3$); if $\left.c\right|_{S}(u)=j$ with $j \in \{1,4\}$, then we color $H$ by $[4] \setminus \{j\}$ such that each vertex has a different color, where $u \in \{v_1,v_2\}$ is a vertex joined to a $C_3$-component $H$ of $G-V(C')$ in $G[S]$. By Lemma \ref{lem3.2}, $c:=\left.c\right|_{S} \cup \left.c\right|_{G''}$ is a 4-CI-coloring of $G$, and $G$ is 4-CI-colorable.
\end{proof}

\begin{claim}\label{claim3.7}
	Each $K_3$ in $G$ is separating.
\end{claim}

\begin{proof}
	Suppose to the contrary that there exists a non-separating $K_3$ in $G$, denoted by $H$. Let $G':=G-V(H)$. Clearly, $G'$ is connected and $G' \ncong (\emptyset,\emptyset)$ (since $n \geq 4$). Since $G$ is connected, let $uv \in E(G)$ for some $u \in V(H)$ and $v \in V(G')$. If $G' \cong C_3$, then let $N_H(u)=\{u_1,u_2\}$ and $N_{G'}(v)=\{v_1,v_2\}$. Define a coloring $c$ of $G$ as follows:  $c(u)=4$, $c(v)=3$, $c(u_2)=c(v_2)=2$ and $c(u_1)=c(v_1)=1$. It is easy to check that $c$ is a 4-CI-coloring of $G$, and thus, $G$ is 4-CI-colorable.
	
	If $G' \ncong C_3$, then by the minimality of $G$, $G'$ admits a 4-CI-coloring $\left.c\right|_{G'}$. W.l.o.g., we may assume that $\left.c\right|_{G'}(v)=4$. Now we use $[3]$ to color $V(H)$ such that each vertex has a different color, and denote $\left.c\right|_{H}$ and $c:=\left.c\right|_{G'} \cup \left.c\right|_{H}$ by the present colorings of $H$ and $G$. Clearly, for each $i \in [3]$, the color class $D_i$ of $G$ is a cycle isolating set of $G$, and the color class $D_{4}$ is not a cycle isolating set of $G$ if and only if $G-N[D_{4}]$ contains a cycle $C$ formed by $\{u_1,u_2\}$ and a path of $G'$, where $N_H(u)=\{u_1,u_2\}$. We may assume that $C$ is a shortest cycle containing $\{u_1,u_2\}$ in $G$. Then $C$ is induced. By Claims \ref{claim3.2} and \ref{claim3.4}-\ref{claim3.6}, $|V(C)|=3$. However, there exists a cycle of length $4$ formed by $H$ and $C$, a contradiction to Claim \ref{claim3.1}. Therefore, $D_4$ is also a cycle isolating set of $G$, and further, we know that $c$ is a 4-CI-coloring of $G$ and $G$ is 4-CI-colorable.
\end{proof}

Let $DK_3$ and $K_3^+$ be the graphs defined as in Section \ref{sec2}.

\begin{claim}\label{claim3.8}
	$G$ contains no $DK_3$.
\end{claim}

\begin{proof}
	Suppose to the contrary that there exists a $DK_3$ in $G$, denoted by $H$. Since $\Delta(G) \leq 3$, and by Claim \ref{claim3.1}, we know that $H$ is induced. Let $G':=G-V(H)$. We straightforward color $H$ using the coloring noted in the proof (for the situation of $G' \cong C_3$) of Claim \ref{claim3.7}, and denote $\left.c\right|_{H}$ by the present coloring of $H$. By Claim \ref{claim3.7}, Lemma \ref{lem3.3} and the minimality of $G$, $G'$ contains no $C_3$-component, and $G'$ admits a 4-CI-coloring $\left.c\right|_{G'}$. Let $c:=\left.c\right|_{G'} \cup \left.c\right|_{H}$. As noted in the proof (for the situation of $G' \ncong C_3$) of Claim \ref{claim3.7}, $u_1u_2$ is contained in no cycle of $G$ formed by $\{u_1,u_2\}$ and a path of $G'$, and $v_1v_2$ is contained in no cycle of $G$ formed by $\{v_1,v_2\}$ and a path of $G'$, where $u_1u_2,v_1v_2 \in E(H)$ and $d_H(v)=2$ for each $v \in  \{u_1,u_2,v_1,v_2\} \subset V(H)$. Thus, $c$ is a 4-CI-coloring of $G$, and $G$ is 4-CI-colorable.
\end{proof}

\begin{claim}\label{claim3.9}
	$G$ contains no $K_3^+$.
\end{claim}

\begin{proof}
	Suppose to the contrary that there exists a $K_3^+$ in $G$, denoted by $H$. Since $\Delta(G) \leq 3$ and by Claim \ref{claim3.1}, $H$ is induced. Let $G':=G-V(H)$, and $V(H)=\{u_1,u_2,u_3,u_4\}$ with $d_H(u_1)=1$ and $d_H(u_2)=3$. By Claim \ref{claim3.7}, Lemma \ref{lem3.3} and the minimality of $G$, $G'$ contains no $C_3$-component, and $G'$ admits a 4-CI-coloring $\left.c\right|_{G'}$. We straightforward present a coloring $\left.c\right|_{H}$ of $H$ as follows: $\left.c\right|_{H}(u_i)=i$ for each $i \in [4]$. Let $c:=\left.c\right|_{G'} \cup \left.c\right|_{H}$. It is easy to see that the color class $D_2$ of $G$ is a cycle isolating set of $G$. The color class $D_{1}$ is not a cycle isolating set of $G$ if and only if $G-N[D_{1}]$ contains a cycle $C$ formed by $\{u_3,u_4\}$ and a path of $G'$. As noted in the proof (for the situation of $G' \ncong C_3$) of Claim \ref{claim3.7}, $u_3u_4$ is contained in no such cycle $C$ of $G$. By the symmetry of $u_3$ and $u_4$, the color class $D_{3}$ or $D_4$ is not a cycle isolating set of $G$ if and only if $G-N[D_{3}]$ or $G-N[D_{4}]$ contains a cycle $C'$ formed by $u_1$ and a path of $G'$. Due to the same reason noted in the proof of Claim \ref{claim3.7}, $u_1$ in fact is contained in no such cycle $C'$ of $G$, unless $|V(C')|=3$. However, there exists a $DK_3$ formed by $H$ and $C'$ in $G$, a contradiction to Claim \ref{claim3.8}. Therefore, for each $i \in [4]$, the color class $D_i$ is a cycle isolating set of $G$, and $c$ is 4-CI-coloring of $G$ and $G$ is 4-CI-colorable.
\end{proof}

	By Claim \ref{claim3.1}, $G$ contain no a $C_4$, and by Claim \ref{claim3.9}, $G$ contain no a $K_3^+$. For each vertex $v \in V(G)$ with $d(v)=\Delta(G)=3$, $G[N[v]] \cong K_{1,3}$, a contradiction to the assumption that $G$ is claw-free. Therefore, the minimal counterexample $G$ does not exist. This completes the proof of Theorem \ref{th3.1}. \qed

\section {\large Closing remarks and open problems}

For certain type of isolation in graphs, the corresponding result of partition is stronger than the result of upper bound. In this note, we would like to prove two results of isolation partition (namely, Conjectures \ref{conj1.7} and \ref{conj1.8}), but failed. The obstacle of our proofs is the maximum degree of a graph. Adding the condition of ``planar'', we wonder whether the following conjecture is true or not.

\begin{conjecture}\label{conj4.1}
	Every connected planar graph, except $C_3$, can be partitioned into four disjoint cycle isolating sets.
\end{conjecture}

A recent nice book \cite{Haynes2023} gave a survey (namely, Chapter 12) on domination partitions of graphs. Cockayne and Hedetniemi \cite{Cockayne1975,Cockayne1977} introduced the domatic number of graphs. In this paper, we shall further define the $\mathcal{H}$-isomatic number of graphs.

\begin{itemize}
	\item A {\it domatic $k$-partition} of a graph $G$ is a partition of $V(G)$ into $k$ sets $V_1,V_2,\ldots,V_{k}$ such that for each $i \in [k]$, $V_i$ is a dominating set of $G$. The {\it domatic number} of $G$, denoted by ${\rm dom}(G)$, is the maximum positive integer $k$ of any domatic $k$-partition of $G$.
	
	\item An {\it $\mathcal{H}$-isomatic $k$-partition} of a graph $G$ is a (weak) partition of $V(G)$ into $k$ sets $V_1,V_2,\ldots,V_{k}$ such that for each $i \in [k]$, $V_i$ is an $\mathcal{H}$-isolating set of $G$. The {\it $\mathcal{H}$-isomatic number} of $G$, denoted by ${\rm iso}(G,\mathcal{H})$, is the maximum $k$ of any $\mathcal{H}$-isomatic $k$-partition of $G$. If $\mathcal{H}=\{H\}$, then the notation is defined as {\it $H$-isomatic $k$-partition/number} simply, and write ${\rm iso}(G,H)$ for short.

\end{itemize}

Note that a $K_1$-isomatic $k$-partition is just a domatic $k$-partition of a graph $G$, and that ${\rm dom}(G)={\rm iso}(G,K_1)$. We simply say that an {\it $\mathcal{H}$-isomatic partition} is an $\mathcal{H}$-isomatic $k$-partition for some unspecifed integer $k \geq 1$. A $K_2$-isomatic partition of a graph $G$ is simply called an {\it isomatic partition} of $G$, and we denote by ${\rm iso}(G)$ instead of ${\rm iso}(G,K_2)$ the {\it isomatic number} of $G$. Theorems \ref{th1.1} and \ref{th1.6} imply that ${\rm dom}(G) \geq 2$ for any connected graph $G \ncong K_1$, and that ${\rm iso}(G) \geq 3$ for any connected graph $G \notin \{K_2,C_5\}$, respectively.

\begin{proposition}
	If $G$ is a graph of order $n$, then ${\rm iso}(G,H) \leq \frac{n}{\iota(G,\mathcal{H})}$.
\end{proposition}

\begin{proof}
	Let $\{V_1,V_2,\ldots,V_k\}$ be an $\mathcal{H}$-isomatic $k$-partition of $G$ where $k={\rm iso}(G,\mathcal{H})$. Since, for each $i \in [k]$, $V_i$ is an $\mathcal{H}$-isolating set of $G$, we have $|V_i| \geq \iota(G,\mathcal{H})$. Then
	$$n=\sum_{i=1}^{k}|V_i| \geq \sum_{i=1}^{k}\iota(G,\mathcal{H})=k\iota(G,\mathcal{H})={\rm iso}(G,\mathcal{H}) \cdot \iota(G,\mathcal{H}).$$
	The result follows.
\end{proof}

For any integer $k \geq 1$, we define the {\it $k$-clique isomatic $k$-partition} and {\it cycle isomatic $k$-partition} of a graph $G$, and denote by ${\rm iso}(G,k)$ and ${\rm iso}_c(G)$ the {\it $k$-clique isomatic number} and {\it cycle isomatic number} of $G$, respectively. Although Conjectures \ref{conj1.7} and \ref{conj1.8} have not been proven yet, the following two results of partition are easy to obtain.

\begin{proposition}\label{prop4.3}
	Every graph can be partitioned into $k$ disjoint $k$-clique isolating sets, that is, every graph admits a $k$-clique isomatic $k$-partition.
\end{proposition}

\begin{proof}
	Let $G$ be a counterexample of minimum order $|V(G)|$. Then, $G$ can not be partitioned into $k$ disjoint $k$-clique isolating sets. Let $\pi=\{V_1,V_2,\ldots,V_k\}$ be a vertex partition of $V(G)$ into $k$ pairwise disjoint subsets. If $G$ contains no $K_k$, then let $V_k=V(G)$ and $V_i=\emptyset$ for each $i \in [k-1]$. It is easy to see that $\pi$ is a $k$-clique isomatic $k$-partition of $G$ now, a contradiction. Assume that $G$ contains a $K_k$, denoted by $H$. Let $V(H)=\{u_1,u_2,\ldots,u_k\}$ and $G'=G-V(H)$. By the minimality of $G$, $G$ admits a $k$-clique isomatic $k$-partition $\pi'=\{V_1',V_2',\ldots,V_k'\}$. For each $i\in [k]$, let $V_i=V_i' \cup \{u_i\}$. It is easy to check that $V_i$ is a $k$-clique isolating set of $G$. Thus, $\pi$ is a $k$-clique isomatic $k$-partition of $G$, also a contradiction.
\end{proof}

\begin{proposition}\label{prop4.4}
	Every graph can be partitioned into three disjoint cycle isolating sets, that is, every graph admits a cycle isomatic $3$-partition.
\end{proposition}

\begin{proof}
	Let $G$ be a counterexample of minimum order $|V(G)|$. It is clear that $G$ is connected; otherwise, $G$ has a component that can not be partitioned into three disjoint cycle isolating sets, contradicting the minimality of $G$. Let $\pi=\{V_1,V_2,V_3\}$ be a vertex partition of $V(G)$ into three pairwise disjoint subsets. If $G$ is a tree, then let $V_3=V(G)$ and $V_i=\emptyset$ for each $i \in [2]$. Clearly, $\pi$ is a cycle isomatic $3$-partition of $G$, a contradiction. If $G$ is a cycle, then $|V(G)| \geq 3$ and let $v_1,v_2,v_3 \in V(G)$. For each $i \in [3]$, let $v_i \in V_i$, and we partition $V(G) \setminus \{v_1,v_2,v_3\}$ arbitrarily. It is easy to verify that $\pi$ is a cycle isomatic $3$-partition of $G$, a contradiction. So, we may assume that $G \notin \{K_2,C_5\}$ is a connected graph. By Theorem \ref{th1.6}, we know that $G$ admits an isomatic $3$-partition $\pi^*$. It is easy to see that $\pi^*$ is also a cycle isomatic $3$-partition of $G$, again a contradiction. The result follows.
\end{proof}

By Propositions \ref{prop4.3} and \ref{prop4.4}, we know that for any graph $G$, ${\rm iso}(G,k) \geq k$ and ${\rm iso}_c(G) \geq 3$. Furthermore, we remark that Conjectures \ref{conj1.7} and \ref{conj1.8} are equivalent to the following conjectures (i) and (ii), respectively.

\begin{conjecture}
	(i) Let $k \geq 3$ be an integer. For any connected graph $G \ncong K_k$, ${\rm iso}(G,k) \geq k+1$. (ii) For any connected graph $G \ncong C_3$, ${\rm iso}_c(G) \geq 4$.
\end{conjecture}

At last, we close this paper with a more general problem.

\begin{problem}
	For some finite set $\mathcal{E}(G,\mathcal{H})$ of except graphs, determine the maximum positive integer $k$ such that ${\rm iso}(G,\mathcal{H}) \geq k$ for connected graphs $G \notin \mathcal{E}(G,\mathcal{H})$.
\end{problem}

\vspace{3mm}
\noindent{\large\bf Acknowledgments}
\vspace{3mm}

This work was supported by the National Natural Science Foundation of China (Nos. 12171402 and 12361070).

\end{document}